\documentclass[twoside, 12pt]{amsart}

\usepackage[utf8]{inputenc}
\usepackage{blindtext}
\usepackage[T1]{fontenc}
\usepackage{amsmath}
\usepackage{amssymb}
\usepackage{amsthm}
\usepackage{graphicx}
\usepackage{color,epsfig}      
\usepackage{tikz}
\usetikzlibrary{arrows,shapes} 

\newtheorem{defn}{Definition}
\newtheorem{ex}{Example}
\newtheorem{thm}{Theorem}

\newtheorem{rem}{Remark}
\newtheorem{prop}{Proposition}
\newtheorem{cor}{Corollary}
\newtheorem{lem}{Lemma}
\newtheorem{prob}{Problem}

\DeclareMathOperator{\inter}{int}
\DeclareMathOperator{\vol}{vol}

\DeclareMathOperator{\bd}{bd}
\DeclareMathOperator{\conv}{conv}

\DeclareMathOperator{\ir}{ir}
\DeclareMathOperator{\cirr}{cr}
\DeclareMathOperator{\skel}{skel}

\DeclareUnicodeCharacter{2212}{-}
\DeclareUnicodeCharacter{0301}{\'{e}}

\newcommand{\Sph}{\mathbb{S}}
\newcommand{\HH}{\mathbb{H}}
\renewcommand{\Re}{\mathbb{R}}
\newcommand{\Eu}{\mathbb{E}}
\newcommand{\M}{\mathbb{M}}
\newcommand{\BB}{\mathbf{B}}

\begin{document}

\title[Discrete isoperimetric problems]{Discrete isoperimetric problems in spaces of constant curvature}

\author[B. Basit]{Bushra Basit}
\author[Z. L\'angi]{Zsolt L\'angi}

\address{Bushra Basit, Department of Geometry, Budapest University of Technology and Economics,\\
M\H uegyetem rkp. 3., H-1111 Budapest, Hungary} 
\email{bushrabasit18@gmail.com}
\address{Zsolt L\'angi, Department of Geometry, Budapest University of Technology and Economics, and MTA-BME Morphodynamics Research Group,\\
M\H uegyetem rkp. 3., H-1111 Budapest, Hungary} 
\email{zlangi@math.bme.hu}

\thanks{Partially supported by the BME Water Sciences \& Disaster Prevention TKP2020 Institution Excellence Subprogram, grant no. TKP2020 BME-IKA-VIZ and
the NKFIH grant K119670.}


\begin{abstract}
The aim of this paper is to prove isoperimetric inequalities for simplices and polytopes with $d+2$ vertices in Euclidean, spherical and hyperbolic $d$-space.
In particular, we find the minimal volume $d$-dimensional hyperbolic simplices and spherical tetrahedra of a given inradius. Furthermore, we investigate the properties of maximal volume spherical and hyperbolic polytopes with $d+2$ vertices with a given circumradius, and the hyperbolic polytopes with $d+2$ vertices with a given inradius and having a minimal volume or minimal total edge length. Finally, for any $1 \leq k \leq d$, we investigate the properties of Euclidean simplices and polytopes with $d+2$ vertices having a fixed inradius and a minimal volume of its $k$-skeleton. The main tool of our investigation is Euclidean, spherical and hyperbolic Steiner symmetrization.
\end{abstract}

\maketitle

\section{Introduction}\label{sec:intro}

The classical Discrete Isoperimetric Inequality, stating that among convex $n$-gons of unit area in the Euclidean plane $\Eu^2$, the ones with minimal perimeter are the regular ones, was already observed by Zenodorus \cite{Blasjo}. Since that time, many similar problems, often called isoperimetric problems, have appeared in the literature \cite{BMP05}, asking about the extremal value of a geometric quantity in a certain family of convex polytopes. In particular, results about convex polyhedra in Euclidean $3$-space with a given number of vertices or faces and having a fixed inradius or circumradius can be found, e.g. in \cite{ftl} or \cite{bermanhanes}.

Much less is known about convex polytopes with a given number of vertices in the $d$-dimensional Euclidean space $\Eu^d$. Among the known results for polytopes with $d+2$ vertices, which in most problems seems to be the first interesting case, we can mention the paper \cite{BB} of B\"or\"oczky and B\"or\"oczky Jr. finding the minimum surface area polytopes of unit volume in this family, \cite{KW03} of Klein and Wessler determining the maximum volume polytopes with unit diameter, and \cite{AZ16} of G.Horv\'ath and the second named author about maximum volume polytopes with unit circumradius. We note that both \cite{KW03} and \cite{AZ16} contain partial results for polytopes with $d+3$ vertices (see also \cite{KW05}).

In the authors' knowledge, in hyperbolic and spherical spaces there are just a few examples of solutions of isoperimetric problems. In particular, Peyerimhoff \cite{P02} proved that among hyperbolic simplices inscribed in a given ball, the regular ones have maximal volume, and B\"or\"oczky \cite{Boroczky} proved the analogous statement for spherical simplices.

Our goal in this paper is to examine isoperimetric problems in $d$-dimensional Euclidean, hyperbolic and spherical space for polytopes with $d+1$ or $d+2$ vertices.
Our main tool is the Steiner symmetrization of convex sets, whose hyperbolic and spherical counterparts were introduced by Schneider \cite{Schneider} and B\"or\"oczky \cite{Boroczky}, respectively.

In Section~\ref{sec:prelim}, we introduce the necessary notation and prove some lemmas needed in our investigation, describing also the properties of Steiner symmetrization. In Section~\ref{sec:results} we present our main results. In Section~\ref{sec:app} we apply our methods to prove statements for measures generated by rotationally symmetric density functions. Finally, in Section~\ref{sec:remarks} we present some additional remarks and pose some open problems.

\section{Preliminaries}\label{sec:prelim}

Throughout the paper we denote the $d$-dimensional Euclidean, hyperbolic and spherical space by by $\Eu^d, \HH^d$ and $\Sph^d$, respectively. We regard $\Sph^d$ as the unit sphere centered at the origin $o$ of the space $\Eu^{d+1}$.

Let $\M \in \{ \Eu^d, \HH^d, \Sph^d \}$. For any two points $p,q \in \M$, which are not antipodal if $\M = \Sph^d$, we denote the shortest geodesic connecting $p$ and $q$ by $[p,q]$, and call it the \emph{closed segment with endpoints $p$ and $q$}; the distance $d(p,q)$ of $p$ and $q$ is defined as the length of $[p,q]$. If $\M = \Eu^d$, then for any point $x$ its distance from the origin $o$ is the Euclidean norm of $x$, which we denote by $||x||$.

We say that $K \subset \M$ is convex if for any two points $p,q \in K$, we have $[p,q] \subseteq K$; here, if $\M = \Sph^d$, we also assume that $K$ is contained in an open hemisphere of $\Sph^d$. Furthermore, for any $X \subseteq \M$, where $X$ is contained in an open hemisphere if $\M = \Sph^d$, the intersection of all convex sets containing $X$ is a convex set, called the \emph{convex hull} of $X$, and denoted by $\conv(X)$. A compact, convex set with nonempty interior is called a \emph{convex body}. A convex polytope in $\M$ is the convex hull of finitely many points. A face of a convex polytope $P$ is the intersection of the polytope with a supporting hyperplane of $P$, $0$-dimensional faces of $P$ are called vertices of $P$, and we denote the vertex set of $P$ with $V(P)$.

Let $P \subset \M$ be a $d$-dimensional convex polytope. The radius of a largest ball contained in $P$ is called the \emph{inradius} of $P$, and is denoted by $\ir(P)$, and the radius of a smallest ball containing $P$ is called the \emph{circumradius} of $P$, denoted by $\cirr(P)$.
For any $1 \leq k \leq d-1$, the \emph{$k$-skeleton} of $P$ is the union of its $k$-dimensional faces; we denote it by $\skel_k(P)$. For any $1 \leq k \leq d$, we denote $k$-dimensional volume by $\vol_k (\cdot)$, and call $\vol_k(\skel_k(P))$ the \emph{total $k$-content} of $P$, or if $k=1$, then the \emph{total edge length} of $P$.

Note that in the projective ball model of $\HH^n$, hyperbolic line segments are represented by Euclidean line segments, implying that hyperbolic convex sets are exactly those represented by Euclidean convex sets. Furthermore, applying central projection onto a tangent hyperplane of $\Sph^d$, the same holds in spherical space as well. This shows that Remarks~\ref{rem:combin}-\ref{rem:convhull}, which are well known in $\Eu^d$, hold in $\HH^d$ and $\Sph^d$ as well. We note that for $\M=\Eu^d$, Remarks~\ref{rem:combin} and \ref{rem:face_structure} can be found in \cite[Section 6.1]{Grunbaum} (see also \cite{AZ16}).

\begin{rem}\label{rem:combin}
Every $d$-dimensional simplicial polytope $P  \subset \M$ with $d+2$ vertices is the convex hull of two simplices $S_1, S_2$ with dimensions $\dim S_1 = k$ and $\dim S_2 = d-k$ with some $0 \leq k \leq \lfloor \frac{d}{2} \rfloor$, respectively, such that $S_1 \cap S_2 = \{ p \}$ for some point $p$ in the relative interior of both $S_1$ and $S_2$. Furthermore, every $d$-dimensional polytope $Q \subset \M$ is the convex hull of $r$ points and a $(d-r)$-dimensional simplicial polytope with $d-r+2$ vertices, for some $0 \leq r \leq d-2$.
\end{rem}

We note that in Remark~\ref{rem:combin}, $S_1$ and $\conv(S_2 \cup Q)$ are two simplices which intersect at a point which is a relative interior point of $S_1$ and a point of $\conv(S_2 \cup Q)$.

\begin{rem}\label{rem:face_structure}
The combinatorial structure of a convex polytope with a few vertices can be described e.g. by a \emph{Gale diagram} of the polytope. If $P$ is a $d$-dimensional convex polytope with $d+2$ vertices, and $P= \conv (S_1 \cup S_2 \cup \{ p_1, \ldots, p_r\})$, where $\dim S_1 = k_1$, $\dim S_2 = k_2$ with $k_1+k_2+r=d$ and $k_1, k_2 \geq 1$, then $P$ can be represented by a multiset consisting of the points $-1,0,1$ on the real line, with multiplicities $k_1+1, r, k_2+1$. Here the elements equal to $-1$ represent the vertices of $S_1$, those equal to $1$ represent the vertices of $S_2$, and those equal to $0$ represent the $p_i$. In this representation, a subset of the vertex set of $P$ is the vertex set of a face of $P$ if and only if the complement of the vertex set of $P$ is associated to points in the diagram whose convex hull contains $0$ in its relative interior. In particular, if $k_1, k_2 \geq 2$, then every pair of vertices of $P$ is connected by an edge, and the same holds for any pair apart from the vertices of $S_1$ and/or $S_2$ if $k_1=1$ and/or $k_2 = 1$, respectively.
\end{rem}

\begin{rem}\label{rem:intersection}
Let $P = \conv (S_1 \cup S_2)$ be the convex hull of two simplices with $\dim S_1 + \dim S_2 = d$ such that the subspaces generated by them intersect at a single point $p$ which belongs to $S_1 \cap S_2$, and is different from any vertex of $S_1$ and $S_2$. Then, for any vertices $q_1, q_2$ of $S_1$, there is a hyperplane $H$ containing $V(P) \setminus \{ q_1, q_2 \}$ and intersecting $[q_1,q_2]$.
\end{rem}

\begin{proof}
Observe that $V(P) \setminus \{ q_1, q_2\}$ contains $d$ points, and thus, there is a hyperplane containing it. We show that this hyperplane can be chosen in such a way that it intersects $[q_1,q_2]$. Indeed, any hyperplane $H'$ containing $V(P) \setminus \{ q_1, q_2 \}$ satisfies $S_2 \subset H'$, implying also $p \in H'$. Furthermore, since $V(S_1) \setminus \{ q_1, q_2 \} \subset H' \cap L_1$, if $p \notin \conv (V(S_1) \setminus \{ q_1, q_2 \})$, then there is a unique hyperplane $H'$ containing $V(P) \setminus \{ q_1, q_2 \}$, and this hyperplane intersects $[q_1,q_2]$. If $p \in \conv (V(S_1) \setminus \{ q_1, q_2 \})$, then the subspace spanned by $V(P) \setminus \{ q_1, q_2 \}$ is $(d-2)$-dimensional, and thus, we may choose $H'$ such that it contains a point of $[q_1,q_2]$.
\end{proof}

\begin{rem}\label{rem:convhull}
Let $H$ be a hyperplane of $\M$. Let $Q$ be a compact, convex set in $H$, and let $[p_1,p_2]$ be a segment such that $H \cap [p_1,p_2]$ is a singleton $\{p\}$. Let $P = \conv (Q \cup [p_1,p_2])$. Then, for any plane $F$ containing $[p_1,p_2]$, $F \cap P$ is either $[p_1,p_2]$, or a triangle $\conv \{ q, p_1,p_2 \}$, or a quadrangle $\conv \{ q_1, q_2, p_1, p_2 \}$, where $q, q_1, q_2$ are relative boundary points of $Q$.
\end{rem}

\subsection{Euclidean Steiner symmetrization}

\begin{defn}\label{defn:EuSteiner}
Let $H$ be a hyperplane in $\Eu^d$. The \emph{Steiner symmetrization $\sigma$} with respect to $H$ is the geometric transformation which assigns to any convex body $K \subset \Eu^d$ the unique compact set $K'$ symmetric to $H$ with the property that for any line $L$ orthogonal to $H$, $K \cap L$ and $K' \cap L$ are both nondegenerate segments of equal length, or singletons, or the empty set (see Figure~\ref{fig:Eu_Steiner}). In this case we call the set $\sigma(K)$ the \emph{Steiner symmetral} of $K$.
\end{defn}

\begin{figure}[ht]
\begin{center}
\includegraphics[width=0.4\textwidth]{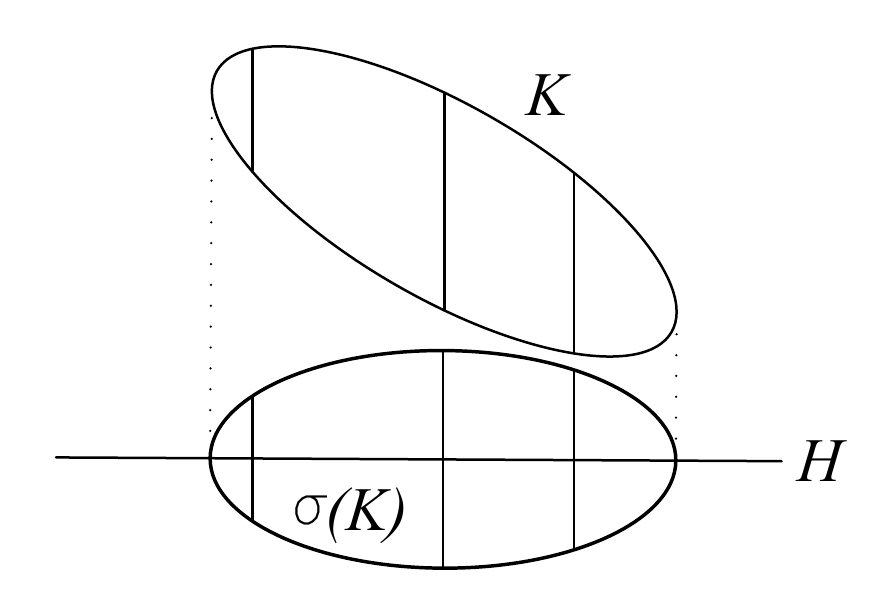}
\caption{Euclidean Steiner symmetrization of a plane convex body $K$ with respect to a horizontal line $H$. Sections with the same vertical lines are of equal length.}
\label{fig:Eu_Steiner}
\end{center}
\end{figure}

\begin{lem}\label{lem:eu_main}
Let $P \subset \Eu^d$ be a convex polytope with vertices $p_1,p_2, \ldots, p_n$. Assume that for some hyperplane $H' \subset \Eu^d$ intersecting $[p_1,p_2]$, we have $p_3, \ldots, p_n \in H'$. Let $H$ be the hyperplane bisecting $[p_1,p_2]$. Let $\pi : \Eu^d \to H$ denote the orthogonal projection onto $H$ and let $\sigma$ denote the Steiner symmetrization with respect to $H$. 
Then
\begin{equation}\label{eq:euStein_vol}
\sigma(P) = \conv ( [p_1,p_2] \cup \{ \pi(p_3), \ldots, \pi(p_n) \} ).
\end{equation}
Furthermore,
\begin{equation}\label{eq:euStein_ir}
\ir(\sigma(P)) \geq \ir(P)
\end{equation}
with equality if and only if $P$ is symmetric to $H$.
Finally, if for some $1 \leq k \leq d-1$, $P$ satisfies the property that for any $ 3 \leq i_1 \leq i_2 \leq \ldots \leq i_m$, $\conv \{ p_{1}, p_{i_1}, \ldots, p_{i_m} \}$ is a $k$-face of $P$ if and only if $\conv \{ p_{2}, p_{i_1}, \ldots, p_{i_m} \}$ is a $k$-face of $P$, then
\begin{equation}\label{eq:euStein_skel}
\vol_k(\skel_k(\sigma(P)) \leq \vol_k(\skel_k(P)).
\end{equation}

\end{lem}

For the proof of Lemma~\ref{lem:eu_main} we need Lemma~\ref{lem:shadow}. Lemma~\ref{lem:shadow} explores the properties of the so-called \emph{linear parameter systems} or \emph{shadow systems}, introduced by Rogers and Shephard in \cite{RS57}. Lemma~\ref{lem:shadow} can be found in \cite{Joos}, and hence, we omit its proof.

\begin{lem}\label{lem:shadow}
Let $p_1, \ldots, p_k, v \in \Eu^d$, and $\lambda_1, \ldots, \lambda_k \in \Re$. For $i=1,2,\ldots,k$ and all $t \in \Re$, set $p_i(t) = p_i + \lambda_i t v$ and
$S(t) = \conv \{ p_1(t), \ldots, p_k(t)  \}$.
Then the function $f: \Re \to \Re$, $f(t) = \vol_{k-1}(S(t))$ is a convex function of $t$, and if the points $p_1, \ldots, p_k, v$ are affinely independent, then $f$ is strictly convex.
\end{lem}

\begin{proof}[Proof of Lemma~\ref{lem:eu_main}]
Clearly, to prove (\ref{eq:euStein_vol}), it is sufficient to prove the first equality, which is the immediate consequence of Remark~\ref{rem:convhull}.
The inequality in (\ref{eq:euStein_ir}) follows from the fact that the Steiner symmetral of a ball of radius $r$ is a ball of radius $r$. The equality case in (\ref{eq:euStein_ir}) is easy to see if $P$ is a triangle in $\Eu^d$, and applying this observation for the sections of $P$ with planes through $[p_1,p_2]$ yields the statement in the general case.

Now we prove (\ref{eq:euStein_skel}). Assume that $P$ is not symmetric to $H$.
Let $v$ be a unit normal vector of $H$, and for any $p_i$, where $i \geq 3$, let $d_i$ denote the signed distance of $p_i$ from $H$ in such a way that $v$ points towards the positive side of $H$. We define the linear parameter system $L(t)= \{ p_i(t) : i=1,2,\ldots, n \}$, where $p_1(t)=p_1$, $p_2(t)=p_2$ and for $i=3,4,\ldots,n$, $p_i(t)= p_i - (1-t) d_i v$, and set $P(t) = \conv (L(t))$.
Observe that $P(1)=P$, $P(-1)$ is the reflected copy of $P$ about $H$, and $P(0)= \conv ([p_1,p_2] \cup \{ \pi(p_i), i=3,4,\ldots n \}$ is the Steiner symmetral of $P$ with respect to $H$. Note that for any value of $t$, the points $p_i(t)$, where $i=3,4,\ldots,n$, lie in a rotated copy of $H'$ around $H \cap H'$. Thus, by Lemma~\ref{lem:shadow}, the function $g: [-1,1] \to \Re$, $g(t) = \vol_k(\skel_k(P(t)))$ is strictly convex. Since $g(-t)=g(t)$ for any $t \in [-1,1]$, it implies that the unique minimum of $g$ is attained at $t=0$. This readily yields (\ref{eq:euStein_skel}).
\end{proof}

\subsection{Hyperbolic Steiner symmetrization}

In the following definition, for a hyperbolic line $L \subset \HH^d$, the line $L$ and the hypercycles with axis $L$ that are contained in a plane through $L$ are called the \emph{$g$-lines of $\HH^d$ with axis $L$} (see \cite{P02}).

\begin{defn}\label{defn:HSteiner}
Let $H$ be a hyperplane in $\HH^d$, and $L$ be a line orthogonal to $H$. The \emph{Steiner symmetrization $\sigma$} with respect to $H$ and with axis $L$ is the geometric transformation which assigns to any convex body $K \subset \HH^d$ the unique compact set $K'$ symmetric to $H$ with the property that for any $g$-line $L'$ with axis $L$, $K \cap L'$ and $K' \cap L'$ are both nondegenerate segments of equal length, or singletons, or the empty set (see Figure~\ref{fig:H_Steiner}). In this case we call the set $\sigma(K)$ the \emph{Steiner symmetral} of $K$.
\end{defn}

\begin{figure}[ht]
\begin{center}
\includegraphics[width=0.4\textwidth]{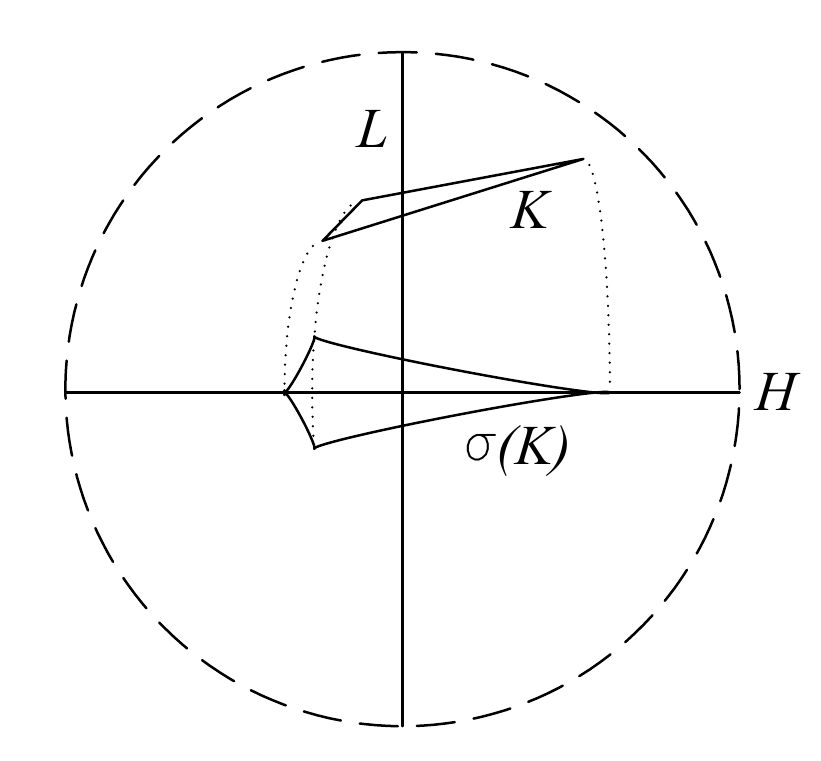}
\caption{Hyperbolic Steiner symmetrization of a triangle $K$ with respect to a horizontal line $H$ and axis $L$ in the projective disk model. Dotted lines denote $g$-lines with axis $L$.}
\label{fig:H_Steiner}
\end{center}
\end{figure}

We note that hyperbolic Steiner symmetrization preserves volume \cite[Proposition 8]{P02}.

In the next lemmas, if $H$ is a hyperplane in $\HH^d$ orthogonal to a line $L$, the \emph{$g$-orthogonal projection $\pi : \HH^d \to H$ onto $H$ with axis $L$} is defined by $\pi(p)=q$ for any $p \in \HH^d$, where $q$ is the intersection point of $H$ and the unique $g$-line through $p$ with axis $L$.
Our next lemma is proved by Peyerimhoff \cite[Proposition 9]{P02}, and hence, we omit its proof.

\begin{lem}\label{lem:hyp_planar}
Let $T = \conv \{ p_1, p_2, p_3 \}$ be a triangle in $\HH^2$. Let $L$ be the line through $[p_1,p_2]$, and let $H$ be the bisector of $[p_1,p_2]$. Let $\pi : \HH^2 \to H$ denote the $g$-orthogonal projection onto $H$ with axis $L$, and let $\sigma$ denote the Steiner symmetrization with respect to $H$ and with axis $L$. Then
\begin{equation}\label{eq:hypStvol}
\sigma(T) \subseteq \conv \{ p_1, p_2, \pi(p_3)\}, \hbox{ and}
\end{equation}
\begin{equation}\label{eq:hypStlen}
d(p_1,p_3) + d(p_2,p_3) \geq d(p_1,\pi(p_3)) + d(p_2,\pi(p_3)).
\end{equation}
Furthermore, in any of (\ref{eq:hypStvol}) or (\ref{eq:hypStlen}), we have equality if and only if $p_3 \in H$.
\end{lem}

\begin{lem}\label{lem:hyp_main}
Let $P \subset \HH^d$ be a convex polytope with vertices $p_1,p_2, \ldots, p_n$. Assume that for some hyperplane $H' \subset \HH^d$ intersecting $[p_1,p_2]$, we have $p_3, \ldots, p_n \in H'$. Let $L$ be the straight line through $[p_1,p_2]$ and let $H$ be the bisector of $[p_1,p_2]$. Let $\pi : \HH^d \to H$ denote the $g$-orthogonal projection onto $H$ with axis $L$, and let $\sigma$ be the Steiner symmetrization with respect to $H$ with axis $L$.
Then
\begin{equation}\label{eq:hypStmain1}
\sigma(P) \subseteq \conv ([p_1,p_2] \cup \{ \pi(p_3), \ldots, \pi(p_n) \} ), 
\end{equation}
\begin{equation}\label{eq:hypStmain2}
\vol_d(P) = \vol_d(\sigma(P)) \leq \vol_d(\conv ([p_1,p_2] \cup \{ \pi(p_3), \ldots, \pi(p_n) \}), \hbox{ and}
\end{equation}
\begin{equation}\label{eq:hypStir}
\ir(P) \leq \ir(\conv ([p_1,p_2] \cup \{ \pi(p_3), \ldots, \pi(p_n) \} )),
\end{equation}
with equality in any of (\ref{eq:hypStmain1}), (\ref{eq:hypStmain2}), or (\ref{eq:hypStir}) if and only if $P$ is symmetric to $H$.
\end{lem}

\begin{proof}
Note that $P = \bigcup_F (F \cap P)$, where $F$ runs over the planes containing $L$, and these intersections are mutually disjoint apart from $[p_1,p_2]$.
Let $F$ be a plane with $L \subset P$. By Remark~\ref{rem:convhull}, $F \cap P = \conv \{ p_1,p_2, q \}$ for some relative boundary point $q$ of $\conv \{ p_3, \ldots, p_n \}$. By Lemma~\ref{lem:hyp_planar}, $\sigma(F \cap P) \subseteq \conv( \sigma([p_1,p_2]) \cup \pi(q) )$. Thus, we have
$\sigma(P) \subseteq \conv( \sigma([p_1,p_2]) \cup \pi (\conv \{ p_3,\ldots, p_n \})$. On the other hand, the relation
$\pi (\conv(p_3,\ldots, p_n)) \subseteq \conv \{ \pi(p_3), \ldots, \pi(p_n)$ follows from \cite[Lemma 10]{P02}, which states that if $X$ is a convex set in $H$ containing the intersection point of $H$ and $L$, then $\pi^{-1}(X)$ is convex. Thus, we have
\[
\sigma(P) \subseteq \conv (\sigma([p_1,p_2]) \cup \{ \pi(p_3), \ldots, \pi(p_n) \} ).
\]
Here, by Lemma~\ref{lem:hyp_planar}, there is equality if and only if all $p_i$ with $i \geq 3$ are contained in $H$. The second inequality in Lemma~\ref{lem:hyp_main} readily follows from the first one and the fact that hyperbolic volume is invariant under Steiner symmetrization \cite{P02}. The inequality in (\ref{eq:hypStir}), together with the equality case, follows from the fact that the Steiner symmetral of a ball of radius $r$ is a ball of radius $r$ \cite[Proposition 8]{P02}.
\end{proof}

\subsection{Spherical Steiner symmetrization}

\begin{defn}\label{defn:sphSteiner}
Let $H$ be a hyperplane in $\Sph^d$.
The \emph{Steiner symmetrization $\sigma$} with respect to $H$ is the geometric transformation which assigns to any convex body $K \subset \Sph^d$, disjoint from the two poles of $H$, the unique compact set $K'$ symmetric to $H$ with the property that for any open half circle $L$ orthogonal to $H$ and ending at the two poles of $H$, $K \cap L$ and $K' \cap L$ are both nondegenerate segments of equal length, or singletons, or the empty set (see Figure~\ref{fig:Sph_Steiner}).  In this case we call the set $\sigma(K)$ the \emph{Steiner symmetral} of $K$.
\end{defn}

\begin{figure}[ht]
\begin{center}
\includegraphics[width=0.4\textwidth]{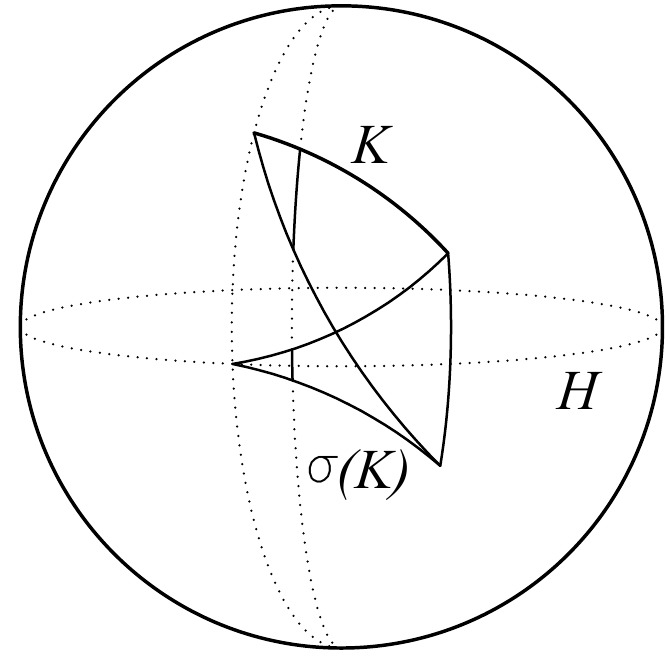}
\caption{Spherical Steiner symmetrization of a triangle $K$ with respect to the spherical line $H$. Dotted lines denote $H$ and half circles perpendicular to it.}
\label{fig:Sph_Steiner}
\end{center}
\end{figure}

We observe that a convex body $K$ disjoint from the poles of $H$ is contained in an open hemisphere whose center lies in $H$; by our definition $\sigma(K)$ is contained in the same open hemisphere. Our next remark is the spherical counterpart of \cite[Lemma 10]{P02}.

\begin{rem}\label{rem:sph_preimage}
Let $H$ be a hyperplane in $\Sph^d$ with poles $\pm p$, and let $\pi:\Sph^d \setminus \{ \pm p\} \to H$ denote the orthogonal projection onto $H$. Then for any convex set $K \subset H$, the set $\pi^{-1}(K)$ is convex.
\end{rem}



The proof of our next lemma can be found in \cite{Boroczky}.

\begin{lem}\label{lem:sph_planar}
Let $H$ be the bisector of the segment $[p_1,p_2] \in \Sph^2$, and let $p_3$ be different from the poles of $H$. Let $\pi$ and $\sigma$ denote the orthogonal projection onto $H$ and the Steiner symmetrization with respect to $H$, respectively. Then
\[
\sigma (\conv\{ p_1,p_2,p_3 \}) \subseteq \conv \{p_1, p_2, \pi(p_3) \},
\]
with equality if and only if $p_3 \in H$. 
\end{lem}

\begin{lem}\label{lem:sph_main}
Let $P \subset \Sph^d$ be a convex polytope with vertices $p_1,p_2, \ldots, p_n$. Assume that for some hyperplane $H' \subset \Sph^d$ intersecting $[p_1,p_2]$, we have $p_3, \ldots, p_n \in H'$. Let $H$ be the bisector of $[p_1,p_2]$, and assume that $P$ is disjoint from the poles of $H$. Let $\pi : \Sph^d \to H$ denote the orthogonal projection onto $H$, and let $\sigma$ be the Steiner symmetrization with respect to $H$.
Then
\begin{equation}\label{eq:sphStmain1}
\sigma(P) \subseteq \conv ([p_1,p_2] \cup \{ \pi(p_3), \ldots, \pi(p_n) \} ), \hbox{ and}
\end{equation}
\begin{equation}\label{eq:sphStmain2}
\vol_d(P) \leq \vol_d(\sigma(P)) \leq \vol_d(\conv ([p_1,p_2] \cup \{ \pi(p_3), \ldots, \pi(p_n) \} ))
\end{equation}
with equality in any of (\ref{eq:sphStmain1}) and (\ref{eq:sphStmain2}) if and only if $P$ is symmetric to $H$. 
\end{lem}

\begin{proof}
By Remark~\ref{rem:convhull} and Lemma~\ref{lem:sph_planar}, $\sigma(P) \subseteq \conv ([p_1,p_2] \cup \pi(P) )$.
On the other hand, by Remark~\ref{rem:sph_preimage}, $\pi(P) \subseteq \conv \{ q, \pi_H(p_3), \ldots, \pi_H(p_n) \}$, where $q$ is the midpoint of $[p_1,p_2]$. Thus, we have $\sigma(P) \subseteq \conv ([p_1,p_2] \cup \{ \pi(p_3), \ldots, \pi(p_n) \} )$, where, by Lemma~\ref{lem:sph_planar}, we have equality if and only if $p_3,\ldots,p_n \subset H$, that is, if $P$ is symmetric to $H$. This, and the fact that spherical volume does not decrease under Steiner symmetrization implies (\ref{eq:sphStmain2}).
\end{proof}

\section{Main results}\label{sec:results}

We start with Theorem~\ref{thm:inscribedvol}, which was proved for $\M = \Eu^d$ in \cite{AZ16}.

\begin{thm}\label{thm:inscribedvol}
Let $\M \in \{ \Eu^d, \HH^d, \Sph^d \}$. Let $B \subset \M$ be a ball, and let $P \subset B$ be a convex polytope with $d+2$ vertices. Then
\[
\vol_d (P) \leq  \vol_d (Q),
\]
where $Q = \conv (S_1 \cup S_2)$ for some regular simplices $S_1, S_2$ inscribed in $B$, with $\dim S_1=k$ and $\dim S_2 = d-k$ for some $1 \leq k \leq \left\lfloor \frac{d}{2} \right\rfloor$,
and contained in mutually orthogonal subspaces of $\M$ intersecting at the center of $B$. Furthermore, for any $d > 0$ there is a value $\varepsilon > 0$ such that
if the radius of $B$ is at most $\varepsilon$, or $\M = \Eu^d$, then $k = \left\lfloor \frac{d}{2} \right\rfloor$.
\end{thm}

\begin{proof}
By compactness argument, in the family of convex polytopes $P \subset B$ with at most $d+2$ vertices, $\vol_d ( \cdot)$ attains its maximum, and any polytope with maximal volume has exactly $d+2$ vertices. Let $P$ be such a polytope. Then, by Remark~\ref{rem:combin}, $P = \conv (S_1 \cup S_2)\subset B$ for some simplices $S_1, S_2$ of dimensions $k$ and $d-k$, respectively, where $1 \leq k \leq \left\lfloor \frac{d}{2} \right\rfloor$, and the subspaces spanned by them intersect in a single common point $p$ of $S_1$ and $S_2$.
Let $L_1$ (resp. $L_2$) denote the $k$-dimensional (resp. $(d-k)$-dimensional) subspace spanned by $S_1$ (resp. $S_2$).
Note that without loss of generality, we may assume that every vertex of $P$ lies on the boundary of $B$.

Consider an edge $[q_1,q_2]$ of $S_1$. We show that $P$ is symmetric to the bisector $H$ of $[q_1,q_2]$.

Suppose for contradiction that $P$ is not symmetric to $H$. Let us choose a hyperplane $H'$ containing $V(P) \setminus \{ q_1, q_2 \}$ such that $H'$ intersects $[q_1,q_2]$, and recall that by Remark~\ref{rem:intersection}, such a hyperplane exists. Let $L$ denote the line through $[q_1,q_2]$.
Note that since $P$ is inscribed in $B$, the bisector $H$ of $[q_1,q_2]$ passes through the center of $B$. Let $\sigma$ denote the Steiner symmetrization with respect to $H$ (or if $\M=\HH^d$, with respect to $H$ and with axis $L$), and let $\pi$ denote the orthogonal projection onto $H$.

Let $P' = \conv ([q_1,q_2] \cup \pi (V(P) \setminus \{ q_1, q_2 \}))$. Then, for $\M = \Sph^d$ and $\M = \HH^d$, by Lemmas~\ref{lem:sph_main} and \ref{lem:hyp_main}, respectively, we have that $\vol_d(P) < \vol_d (P')$, contradicting our assumption that $P$ has maximal volume. In the remaining case, if $\M = \Eu^d$, then $P'$ is a convex polytope in $B$ with $(d+2)$ vertices satisfying $\vol_d(P') = \vol_d(P)$ such that at least one vertex of $P'$ lies in the interior of $B$. Thus, there is a convex polytope $P'' \subset B$ with $(d+2)$ vertices and satisfying $\vol_d(P'') > \vol_d(P)$; a contradiction.

We have shown that $P$ is symmetric to the bisector of the edge $[q_1,q_2]$ of $S_1$. We obtain similarly that $P$ is symmetric to any edge of $S_1$ and $S_2$, implying that $S_1$ and $S_2$ are regular simplices inscribed in $B$ contained in orthogonal subspaces such that the centers of $S_1$, $S_2$ and $B$ coincide.

The fact that if $\M = \Eu^d$, the volume of $\conv (S_1 \cup S_2)$ is maximal if $k= \left\lfloor \frac{d}{2} \right\rfloor$ follows by a simple computation (see, e.g. \cite{AZ16}). The existence of some $\varepsilon > 0$ depending on $d$ only, such that for any ball $B$ of radius at most $\varepsilon$, the volume of $\conv (S_1 \cup S_2)$ is maximal if $k= \left\lfloor \frac{d}{2} \right\rfloor$ follows from the well known observation that in `small scale', hyperbolic and spherical spaces are `almost' Euclidean, and can be proved by a standard limit argument.
\end{proof}

Our next two theorems deal with the volumes of polytopes in Euclidean and hyperbolic spaces, containing a given ball.

\begin{thm}\label{thm:circumscribed_dp1}
Let $\M \in\{ \Eu^d, \HH^d \}$ and let $B \subset \M$ be a ball. Then, among simplices containing $B$, the ones with maximal volume are the regular 
simplices circumscribed about $B$.
\end{thm}

\begin{thm}\label{thm:circumscribed_dp2}
Let $\M \in\{ \Eu^d, \HH^d \}$, let $B \subset \M$ be a ball, and let $P$ be a convex polytope containing $B$ with at most $d+2$ vertices. Then 
$$\vol_d (P) \geq  \vol_d (Q),$$
for a convex polytope $Q = \conv (S_1 \cup S_2)$ containing $B$, where $S_1, S_2$ are regular simplices with $\dim S_1=k$ and $\dim S_2 = d-k$ for some $1 \leq k \leq \frac{d}{2}$,
and contained in mutually orthogonal subspaces of $\M$ intersecting at the center of $B$. Furthermore, for any $d > 0$ there is a value $\varepsilon > 0$ such that
if the radius of $B$ is at most $\varepsilon$, or $\M = \Eu^d$, then $k = \left\lfloor \frac{d}{2} \right\rfloor$, or $d=4$ and $k=1$.
\end{thm}

We note that while the natural choice to solve isoperimetric problems in $\HH^d$ or $\Sph^d$ is to use hyperbolic or spherical Steiner symmetrization, sometimes Euclidean Steiner symmetrization can also be used in these spaces. We illustrate it in the proof of Theorems~\ref{thm:circumscribed_dp1} and \ref{thm:circumscribed_dp2} in hyperbolic space. The proof presented in this paper is a variant of the proof of Lemma 9 in \cite{BL22}.

\begin{proof}[Proof of Theorems~\ref{thm:circumscribed_dp1} and \ref{thm:circumscribed_dp2}]
We start with the proof of Theorem~\ref{thm:circumscribed_dp1}. 

First, assume that $\M = \HH^d$. We represent $\HH^d$ in the projective ball model.
As we already remarked in Section~\ref{sec:prelim}, in this model the points of $\HH^d$ correspond to the points of the open unit ball $\inter (\BB^d)$ of $\Eu^d$ centered at the origin, and a set is convex in $\HH^d$ if and only if it is represented by a convex set in $\inter (\BB^d)$.
The volume element at point $x$ in this model is $\frac{1}{\left( 1+ ||x||^2\right)^{\frac{d+1}{2}}} dx$. Hence, if $A$ is a Borel set in $\inter (\BB^d)$, then
\[
\vol(A) = \int_A \frac{1}{\left( 1+ ||x||^2\right)^{\frac{d+1}{2}}} d x,
\]
where $d x$ denotes integration with respect to $d$-dimensional Lebesgue measure \cite{Ratcliffe}.

Let $S$ be a simplex in $\HH^d$ containing $B$. Without loss of generality, we may assume that the center of $B$ is $o$.
Observe that by compactness, $\vol_d(\cdot)$ attains its minimum on the family of hyperbolic simplices in $\HH^d$. Let $S$ be a simplex in $B$ minimizing
this volume, and let $[p_1,p_2]$ be an edge of $S$. Let $H$ denote the hyperplane perpendicular to the line through $[p_1,p_2]$ that passes through $o$. Let $\sigma$ denote the (Euclidean) Steiner symmetrization with respect to $H$.
Then, clearly, $\sigma(S)$ is a hyperbolic simplex in $\HH^d$ containing $B$. Furthermore, since the density function $\frac{1}{\left( 1+ ||x||^2\right)^{\frac{d+1}{2}}}$ is a strictly increasing function of $||x||$, by Fubini's theorem we have that
\[
\vol_d(\sigma(S)) = \int_{\sigma(S)} \frac{1}{\left( 1+ ||x||^2\right)^{\frac{d+1}{2}}} d x \leq \int_{S} \frac{1}{\left( 1+ ||x||^2\right)^{\frac{d+1}{2}}} d x = \vol_d(S),
\]
with equality if and only if $\sigma(S) = S$, that is, if $S$ is symmetric to $H$. Since $S$ has a minimal hyperbolic volume on the family of simplices containing $B$, it follows that the bisector of any edge of $S$ contains $o$, and thus, $S$ is a regular simplex with $o$ as its center. We obtain also that $S$ is circumscribed about $B$, as otherwise $\vol_d(S)$ is clearly not minimal. This proves Theorem~\ref{thm:circumscribed_dp1} for $\M= \HH^d$. For $\M = \Eu^d$, and to obtain the inequality in Theorem~\ref{thm:circumscribed_dp2}, we can apply an analogous argument.

Now we investigate the equality part of Theorem~\ref{thm:circumscribed_dp2}. By a standard limit argument, it is sufficient to prove the statement for $\M = \Eu^d$, where we may assume that $B$ is the unit ball centered at the origin.

The statement clearly holds if $d=2$, and hence, we assume that $d \geq 3$. Let $P$ be a convex polytope in $\Eu^d$ with $B \subset P$ and minimal volume.
We show that then
\begin{equation}\label{eq:Euvolir}
\vol_d(P) \leq \frac{d^{d/2}}{d!}  \left( \left\lfloor \frac{d}{2} \right\rfloor\right)^{ (\lfloor d/2 \rfloor + 1)/2} \left( \left\lceil \frac{d}{2} \right\rceil\right)^{ (\lceil d/2 \rceil + 1)/2} \left( \left\lfloor \frac{d}{2} \right\rfloor + 1 \right) \left( \left\lceil \frac{d}{2} \right\rceil + 1\right).
\end{equation}
First, an elementary computation shows that the volume of a regular simplex in $\Eu^d$ with unit inradius is $\frac{d^{d/2}(d+1)^{(d+1)/2}}{d!}$, which is strictly smaller than the quantity in (\ref{eq:Euvolir}). Thus, in the following we consider the case that $P$ has exactly $d+2$ vertices.

Similarly like in the proof of Theorem~\ref{thm:circumscribed_dp1}, we have that if $P$ has minimal volume, then $P = \conv (S_1 \cup S_2)$ for two regular simplices $S_1$ and $S_2$, where the affine hulls of $S_1$ and $S_2$ are perpendicular and meet at the common centerpoint of $S_1$ and $S_2$. Let $1 \leq \dim S_1 = k \leq \dim S_2 = d-k \leq d-1$. Without loss of generality, we may assume that $o$ is the center of $S_1$ and $S_2$.

We determine the edge lengths and the dimensions of $S_1$ and $S_2$.
For $i=1,2$, let the inradius of $S_i$ be denoted by $r_i$. Let the vertices of $S_1$ and $p_1, \ldots, p_{k+1}$, and the vertices of $S_2$ be $q_1, \ldots, q_{d-k+1}$.
Note that for any value of $i$, $||p_i||= k r_1$, and $||q_i|| = (d-k) r_2$.
Furthermore, since a facet of $P$ is the convex hull of a facet of $S_1$ and a facet of $S_2$, the volume of $P$ can be obtained as $(k+1)(d-k+1)$ times the volume of $\conv \{ o, p_1, \ldots, p_k, q_1, \ldots, q_{d-k} \}$. In addition, the latter volume is $d!$ times the volume of the parallelotope $P'$ spanned by the vectors $p_1, \ldots, p_k, q_1, \ldots, q_{d-k}$. Let $G$ denote the Gram matrix of these vectors, and note that for any values of $i,j$, $\langle p_i, p_i \rangle = k^2 r_1^2$, $\langle q_i, q_i \rangle = (d-k)^2 r_2^2$ and $\langle p_i, q_j \rangle = 0$, and for any $i \neq j$, $\langle p_i, p_j \rangle = - k r_1^2$ and $\langle q_i, q_j \rangle = - (d-k) r_2^2$. It is well known that $\vol_d (P') = \sqrt{ \det (G)}$. Thus, an elementary computation shows that $\vol_d (P') = r_1^k r_2^{d-k} k^{k+1/2} (d-k)^{d-k+1/2}$, which yields that
\begin{equation}\label{eq:euvol}
\vol_d(P) = \frac{1}{d!} r_1^k r_2^{d-k} k^{k+1/2} (d-k)^{d-k+1/2} (k+1)(d-k+1).
\end{equation}

Note that since the unit ball centered at $o$ touches at least one facet of $P$, it touches every facet of $P$. Let $c_1$ and $c_2$ be the centers of a facet of $S_1$ and a facet of $S_2$, respectively. Then the unit ball centered at $o$  touches the segment $[c_1,c_2]$. Since the triangle $\conv \{ o,c_1,c_2 \}$ is a right triangle with legs of lengths $r_1$ and $r_2$, respectively, we obtain that $(r_1^2 -1)(r_2^2-1) = 1$. Thus, there is a real value $t \in \Re$ such that $r_1 = \sqrt{1+e^t}$ and $r_2 = \sqrt{1+e^{-t}}$. Substituting these expressions into (\ref{eq:euvol}), we obtain that $\vol_d(P) = f_k(t)$, where
\[
f_k(t) = \frac{1}{d!} (1+e^t)^{k/2} (1+e^{-t})^{(d-k)/2} k^{k+1/2} (d-k)^{d-k+1/2} (k+1)(d-k+1).
\]
We find the minimum of $f_k(t)$ for all real $t \in \Re$ and integer $1 \leq k \leq d-1$.
Observe that for any value of $k$, $f_k(t)$ attains its minimum for some value of $t$. On the other hand, an elementary computation shows that the only real solution of the equation $f'_k(t)=0$ is $t=\frac{d-k}{k}$, which shows that for any fixed value of $k$, $f_k(t)$ is minimal if $t=\frac{d-k}{k}$.
Let $g_d(k) = f_k\left( \frac{d-k}{k} \right)$. Then
\[
g_d(k) = \frac{d^{d/2}}{d!} k^{(k+1)/2} (d-k)^{(d-k+1)/2} (k+1) (d-k+1).
\]

We remark that for $d=3$ we have proved the statement, and hence, we assume that $d \geq 4$. Let us extend the domain of $g_d$ to the interval $k \in [1,d-1]$ using the same formula. Then by differentiating twice and applying algebraic transformations, we obtain that the function $\ln g_d$ is strictly convex on $[2,d-2]$, implying that for any integer $2 \leq k < \lfloor \frac{d}{2} \rfloor$, we have $g_d(k) > g_d \left( \lfloor \frac{d}{2} \rfloor \right)$. On the other hand, an elementary computation shows that $g_d(1) \geq g_d(2)$, with equality if and only if $d=4$. This yields the assertion.
\end{proof}

\begin{cor}\label{cor:volinradius}
Let $P$ be a convex polytope in $\Eu^d$ with at most $d+2$ vertices and unit inradius. Then
\[
\vol_d(P) \leq \frac{d^{d/2}}{d!}  \left( \left\lfloor \frac{d}{2} \right\rfloor\right)^{ (\lfloor d/2 \rfloor + 1)/2} \left( \left\lceil \frac{d}{2} \right\rceil\right)^{ (\lceil d/2 \rceil + 1)/2} \left( \left\lfloor \frac{d}{2} \right\rfloor + 1 \right) \left( \left\lceil \frac{d}{2} \right\rceil + 1\right),
\]
with equality if and only if there are regular simplices $S_1, S_2$ lying in orthogonal affine subspaces that meet at the common center of both $S_1$ and $S_2$, such that one of the following holds:
\begin{itemize}
\item $\dim S_1 = \left\lfloor \frac{d}{2} \right\rfloor$ and $\dim S_2 = \left\lceil \frac{d}{2} \right\rceil$;
\item $d=4$, $\dim S_1 = 1$ and $\dim S_2 = 3$.
\end{itemize}
\end{cor}

We are unable to extend Theorem~\ref{thm:circumscribed_dp1} for $\M = \Sph^d$. Nevertheless, we note that the spherical version of Dowker's theorem implies that if $\M = \Sph^2$, then any spherical $m$-gon of minimal area containing the spherical cap $B$ is a regular $m$-gon circumscribed about $P$. We present a proof for $\M = \Sph^3$ using L\'aszl\'o Fejes T\'oth's famous `moment lemma', which we quote as follows (see \cite{FejesToth}).

\begin{lem}[Moment lemma]\label{lem:moment}
Let $\rho(\tau)$ be a strictly increasing function for $0 < \tau < \frac{\pi}{2}$. Furthermore, let $p_1, p_2, \ldots , p_k$ be $k$ points of $\Sph^2$ not all of them lying on a hemisphere, and let $\tau(p)$ denote the minimal spherical distance of the point $p \in \Sph^2$ from a point of the set $\{ p_1, p_2, \ldots, p_k \}$. Let $\bar{\Delta}$ be an equilateral spherical triangle with vertices $\bar{p}_1, \bar{p}_2, \bar{p}_3$ and having area $\frac{2\pi}{k-2}$, and let $\bar{\tau}(p)$ denote the minimum spherical distance of $p$ from one of the vertices of $\bar{\Delta}$. Then we have
\begin{equation}\label{eq:moment}
\int_{\Sph^2} \rho(\tau(p)) \, d \omega_s \leq (2k-4) \int_{\bar{\Delta}} \rho(\bar{\tau}(p)) \, d \omega_s,
\end{equation}
where $d \omega_s$ denotes integration with respect to spherical volume.
Equality in \ref{eq:moment} holds if and only if $\{p_i : i=1,2,\ldots, k\}$ is the vertex set of a regular tetrahedron, octahedron or icosahedron.
\end{lem}

\begin{thm}\label{thm:circumscribed_sph}
Let $B \subset \Sph^3$ be a spherical ball. Let $T \subset \Sph^3$ be a tetrahedron containing $B$, and let $T_{\mathrm{reg}}$ be a regular tetrahedron circumscribed about $B$. Then
\[
\vol_3(T) \geq \vol_3(T_{\mathrm{reg}}),
\]
with equality if and only if $T$ is congruent to $T_{\mathrm{reg}}$.
\end{thm}

\begin{proof}
Clearly, we may assume that $T$ is circumscribed about $B$. Let $p_i$, where $i=1,2,3,4$, denote the points where a face of $T$ touches $B$. Observe that $\bd(B)$ is a $2$-dimensional spherical space, and that the central projections of the faces of $T$ onto $\bd (B)$ coincide with the Voronoi cells of the point set $\{ p_1, p_2, p_3, p_4 \}$; i.e. any point of the cell containing $p_i$ is not farther from $p_i$ than from any other point of the set. Furthermore, if for any $p \in \bd (B)$, $\tau(p)$ denotes the minimum spherical distance from the points $p_i$, then there is a strictly increasing function $\rho(\tau)$ such that the volume of $T$ coincides with
$\int_{\bd(B)} \rho(\tau(p)) \, d \omega_s$. Thus, the assertion follows from Lemma~\ref{lem:moment}.
\end{proof}

Is is well known (see e.g. \cite{Klamkin}) that among simplices $S$ in $\Eu^d$, the ratio of the circumradius of $S$ to its inradius is minimal if and only if $S$ is regular. Our next result is a partial generalization of this for hyperbolic and spherical spaces.

\begin{thm}\label{thm:ratio}
Let $d \geq 2$. For any simplex $S$ in the hyperbolic space $\HH^d$,
\[
\tanh \cirr(S) \geq d \tanh \ir(S),
\]
with equality if, and only if $S$ is regular. Furthermore, for any simplex $S$ in the spherical space $\Sph^d$ with $d \leq 3$,
\[
\tan \cirr(S) \geq d \tan \ir(S),
\]
with equality if, and only if $S$ is regular.
\end{thm}

\begin{proof}
Among simplices $S$ of a given volume in $\HH^d$ or $\Sph^d$, the regular ones have minimal circumradius (see \cite{P02} and \cite{Boroczky}, respectively). Furthermore, among simplices of a given volume in $\HH^d$, the regular ones have maximal inradius (see Theorem~\ref{thm:circumscribed_dp1}), and the same holds for $d \leq 3$ for simplices in $\Sph^d$ (see Theorem~\ref{thm:circumscribed_sph}, and the remark preceding it). Theorem~\ref{thm:ratio} readily follows from these observations.
\end{proof}

We remark that the $d=2$ case of Theorem~\ref{thm:ratio} can be found e.g. in \cite{Veljan}. 
Before stating our next result, we recall that by \cite{P02}, among simplices in $\HH^d$ containing a given ball $B$ the regular simplices circumscribed about $B$ have minimal total edge length. In our next statement we show that there is no similar statement for spherical simplices, and prove a variant of this statement for hyperbolic polytopes with $d+2$ vertices.

\begin{prop}\label{thm:TEL_spherical}
Let $d \geq 2$, let $B$ be a ball of radius $r$ in $\Sph^d$, and let $S_{\mathrm{reg}}$ be a regular simplex circumscribed about $B$. We have the following.
\begin{itemize}
\item[(i)] For any $d \geq 3$, if $r$ is sufficiently close to $\frac{\pi}{2}$, then there is a simplex $S$ containing $B$ and satisfying $\vol_1(\skel_1(S)) < \vol_1(\skel_1(S_{\mathrm{reg}}))$.
\item[(ii)] For any $0 < r < \frac{\pi}{2}$, if $d$ is sufficiently large, then there is a simplex $S$ containing $B$ and satisfying $\vol_1(\skel_1(S)) < \vol_1(\skel_1(S_{\mathrm{reg}}))$.
\end{itemize}
\end{prop}

\begin{proof}
First, we compute $\vol_1(\skel_1(S_{\mathrm{reg}}))$. Let $\pi$ denote the central projection onto the tangent hyperplane $H$ of $\Sph^d$ in $\Eu^{d+1}$ at the center of $S_{\mathrm{reg}}$, and note that $S_{\mathrm{reg}}$ is contained in the open hemisphere centered at this point. Thus, $\pi(S_{\mathrm{reg}})$ is a regular Euclidean simplex with inradius $\tan r$, implying that its circumradius is $d \tan r$. Regarding the center $c$ of $S_{\mathrm{reg}}$ as the origin of the Euclidean $d$-space $H$ and taking the inner product of two vertices of $\pi(S_{\mathrm{reg}})$, we have that if $p_i$ and $p_j$ are two vertices of $S_{\mathrm{reg}}$, then the angle of the Euclidean triangle with vertices $c$, $\pi(p_i)$ and $\pi(p_j)$ at the vertex $c$ is $\arccos \left( - \frac{1}{d} \right)$. This quantity is equal to the angle of the sperical triangle with vertices $p_i, p_j$ and $c$ at $c$. Since the length of the edges $[c,p_i]$ and $[c,p_j]$ in this triangle is $\arctan (d \tan r)$, by the spherical Law of Cosines and using trigonometric identities we obtain that the length of the edges of $S_{\mathrm{reg}}$ is $\arccos \frac{1-d \tan^2 r}{1+d^2 \tan^2 r}$. Thus, $\vol_1(\skel_1(S_{\mathrm{reg}}))$ is equal to $f(r,d) = \binom{d+1}{2} \arccos \frac{1-d \tan^2 r}{1+d^2 \tan^2 r}$.

We show that for any $\varepsilon > 0$, there is a simplex $S$ containing $B$ with $\vol_1(\skel_1(S)) < d \pi+\varepsilon$. Let $S_0$ be a regular $(d-1)$-dimensional simplex whose total edge length is strictly less than $\varepsilon$. We assume that the spherical segment connecting the center $c'$ of $S_0$ to $c$ is perpendicular to $S_0$, and denote the length of this segment by $x$. Assume that $x < \frac{\pi}{2}$, and let $p$ be a point on the line through this segment such that $c$ is the midpoint of $[c',p]$. Observe that if $x$ is sufficiently close to $\frac{\pi}{2}$, then $S= \conv (S_0 \cup \{ p \})$ contains $B$. Furthermore, $\vol_1(\skel_1(S)) < d \pi+\varepsilon$.

Thus, to show the statement it is sufficient to show that if $d \geq 3$ is fixed and $\tan r$ is sufficiently large, or if $0 < r < \frac{\pi}{2}$ is fixed and $d$ is sufficiently large, then $f(d,r) - d  \pi$ is positive, which can be checked by an elementary computation.
\end{proof}

Before Theorems~\ref{thm:edgelength_HH}-\ref{thm:kcontentinradius}, we remark that unlike volume and surface area, the quantity $\vol_k(\skel_k(P))$ is not continuous in the family of $d$-dimensional convex polytopes, if $1 \leq k \leq d-2$.

\begin{thm}\label{thm:edgelength_HH}
Let $B \subset \HH^d$ be a ball with center $c$, and let $P$ be a convex polytope containing $B$ and with $V(P) \leq d+2$ such that $P$ has minimal total edge length among the polytopes in $\HH^d$ containing $B$ and having at most $d+2$ vertices. Then we have one of the following.
\begin{itemize}
\item[(i)] There are regular simplices $S_1$ and $S_2$ with $\dim S_1 + \dim S_2 = d$ such that the subspaces spanned by $S_1$ and $S_2$ are orthogonal complements of each other, $S_1 \cap S_2 = \{ c \}$, and $P = \conv (S_1 \cup S_2)$.
\item[(ii)] There are perpendicular segments $S_1, S_2$ meeting at the common midpoint $q_1$ of both segments, and a regular $(d-2)$-dimensional simplex $Q$ with center $q_2$ such that $c$ lies on the segment $[q_1,q_2]$, $Q$ is perpendicular to $[q_1,q_2]$, and $P= \conv (S_1 \cup S_2 \cup Q)$.
\item[(iii)] $P$ is a regular simplex circumscribed about $B$.  
\end{itemize}
\end{thm}

\begin{proof}
We note that if $V(P)=d+1$, then the statement follows from \cite[Theorem 3]{P02}. Thus, we assume that $V(P)=d+2$.
By Remark~\ref{rem:combin}, there are simplices $S_1$ and $S_2$ with $\dim S_i = k_i$ for $i=1,2$ and intersecting at a singleton which is a relative interior point of both, and there are points $p_1, p_2, \ldots, p_r$ such that $P = \conv(S_1 \cup S_2 \cup \{ p_1, \ldots, p_r \} )$, and $k_1+k_2+r=d$.
If $r > 0$, let $Q = \conv \{ p_1, \ldots, p_r \}$.

Let $[q_1,q_2]$ be any edge of $S_1$. Then there is a unique hyperplane $H'$ containing all other vertices of $P$ such that $H'$ intersects $[q_1,q_2]$ at a point $q$. Let $L$ be the line containing $[q_1,q_2]$, and let $H$ denote the bisector of $[q_1,q_2]$. Let $\sigma$ denote Steiner symmetrization with respect to $H$ with axis $L$. We show that if $P$ is not symmetric to $H$, then $\conv (\sigma(P))$ is a convex polytope with $\ir(P) \leq \ir(\conv (\sigma(P)))$ and strictly smaller total edge length. Here, the first property follows from (\ref{eq:hypStir}). Furthermore, if $k_2 > 1$, then any pair of vertices not in $S_1$ are connected by an edge, which, by Lemma~\ref{lem:hyp_planar} and the fact the length of a segment does not increase under $g$-orthogonal projection, implies the second property. Finally, assume that $k_2 = 1$, i.e $S_2 = [w_1,w_2]$. By Lemma~\ref{lem:hyp_planar}, it is sufficient to show that $[\pi(w_1), \pi(w_2)]$ is not an edge of $\sigma(P)$, where $\pi : \HH^d \to H$ denotes $g$-ortogonal projection onto $H$ with axis $L$.

Observe that $ \conv (V(P) \setminus \{ w_1, w_2 \})$ intersects $S_2$ in a relative interior point of $S_2$. Let this point be $w$. The line $L' \subset H'$ through $q$ and $w$ intersects $\bd (P)$ in $q$ and another point, which we denote by $w'$. Then, by \cite[Lemma 10]{P02} (see also the proof of Lemma~\ref{lem:hyp_main}), $\pi(\conv \{ w_1, w_2, q, w'\})$ is a nondegenerate quadrangle in $H$ with $\conv \{ \pi(w_1), \pi(w_2) \}$ as a diagonal. This, as $\pi(w_1)$, $\pi(w_2)$ are vertices of $\sigma(P)$, yields also that the segment connecting them is not an edge of $\sigma(P)$, and thus, if $P$ is not symmetric to $H$, $\vol_1(\skel_1(\sigma(P))) < \vol_1(\skel_1(P))$.

We have shown that $P$ is symmetric to the bisector of every edge of $S_1$. Note that if $r > 0$ and $k_1 > 1$, then $S'=\conv (S_1 \cup Q')$ is a simplex intersecting $S_2$ at a relative interior point of $S_2$, and thus, by the same argument, $P$ is symmetric to the bisector of every edge of $S'$. Thus, if $k_1 > 1$ or $k_2 > 1$, then $r=0$, and the properties in (i) are satisfied for $P$. Furthermore, if $k_1 = k_2 = 1$, then we obtain similarly that $P$ is symmetric to the bisector of $S_1$, $S_2$ and every edge of $Q$, implying the properties in (ii).
\end{proof}

We note that the statement of Theorem~\ref{thm:kcontentvol} for simplices can be found in \cite{Joos}. 

\begin{thm}\label{thm:kcontentvol}
For any $1 \leq k \leq d-1$, among simplices in $\Eu^d$ with unit volume, the regular ones have minimal total $k$-content. Furthermore, for any $1 \leq k \leq d-1$, among the convex polytopes in $\Eu^d$ with at most $d+2$ vertices and unit volume, if $P$ has minimal total $k$-content, then it satisfies one of the following.
\begin{itemize}
\item[(i)] There are regular simplices $S_1$ and $S_2$ with $\dim S_1 + \dim S_2 = d$ such that the subspaces spanned by $S_1$ and $S_2$ are orthogonal complements of each other, $S_1 \cap S_2 = \{ c \}$, and $P = \conv (S_1 \cup S_2)$.
\item[(ii)] There are regular simplices $S_1, S_2$, of dimensions $\dim S_1 = k_1 \leq k$ and $\dim S_2 = k_2 \leq k$ and lying in orthogonal subspaces that meet at the common center $q_1$ of both simplices, and a regular $(d-k_1-k_2)$-dimensional simplex $Q$ with center $q_2$ such that $Q$ is perpendicular to the segment $[q_1,q_2]$, and $P= \conv (S_1 \cup S_2 \cup Q)$.
\item[(iii)] $P$ is a regular simplex of unit volume.
\end{itemize}
\end{thm}

\begin{proof}
Let $P$ be a simplex of unit volume in $\Eu^d$, and let $[p_1,p_2]$ be an edge of $S$. Let $H$ denote the bisector of $[p_1,p_2]$, and let $\sigma$ denote Steiner symmetrization with respect to $H$. If $P$ is not symmetric to $H$, then $\sigma(P)$ is a unit volume simplex with strictly smaller total $k$-content. Thus, a simplex with minimal total $k$-content is symmetric to the bisector of any of its edges, implying that it is regular.

Now, let $P$ be a $d$-dimensional convex polytope of unit volume and having exactly $d+2$ vertices. Then there are simplices $S_1$, $S_2$ of dimensions $k_1$ and $k_2$ respectively, and possibly a simplex $Q = \conv \{ p_1, p_2, \ldots, p_r \}$ such that the affine hulls of $S_1$ and $S_2$ intersect in a singleton $q_1$ which lies in the relative interior of both $S_1$ and $S_2$, and $k_1+k_2+r=d$.

Consider the case that one of $k_1$ and $k_2$, say $k_1$, is greater than $k$. Let $[p_1,p_2]$ be an edge of $\conv (S_1 \cup Q)$. It is easy to see that any face of $P$ not containing $\conv (S_1 \cup S_2)$ is a simplex. Thus, any $k$-face of $P$ containing at most one of $p_1$ and $p_2$ is a simplex. Furthermore, by the properties of the Gale diagram of $P$ described in Remark~\ref{rem:face_structure}, for any $k$-element subset $S'$ of the vertices of $P$ with $p_1,p_2 \notin S'$, we have that $\conv (\{ p_1 \} \cup S')$ is a $k$-face of $P$ if and only if $\conv (\{ p_2 \} \cup S')$ is a $k$-face of $P$. Thus, we can apply Lemma~\ref{lem:eu_main}, and conclude that if $P$ has minimal $k$-content, then it is symmetric to the bisector of $[p_1,p_2]$. Applying the same argument for all edges of $\conv (S_1 \cup Q)$ we obtain that $\conv (S_1 \cup Q)$ is a regular simplex and the affine hull of $S_2$ meets it at its center, showing that in this case $Q = \emptyset$.

Using the same argument, we obtain that for any values of $k_1$ and $k_2$, if $P$ has minimal total $k$-content, then it is symmetric to the bisector of every edge of $S_1$, $S_2$, and if it exists, of $Q$. This yields the assertion.
\end{proof}

By the idea in the proof of Theorem~\ref{thm:kcontentvol},  we obtain Theorem~\ref{thm:kcontentinradius}.

\begin{thm}\label{thm:kcontentinradius}
For any $1 \leq k \leq d-1$, among simplices in $\Eu^d$ with unit inradius, the regular ones have minimal total $k$-content. Furthermore, for any $1 \leq k \leq d-1$, among the convex polytopes in $\Eu^d$ with $d+2$ vertices and unit inradius, if $P$ has minimal total $k$-content, then it satisfies one of the following.
\begin{itemize}
\item[(i)] There are regular simplices $S_1$ and $S_2$ with $\dim S_1 + \dim S_2 = d$ such that the subspaces spanned by $S_1$ and $S_2$ are orthogonal complements of each other, $S_1 \cap S_2 = \{ c \}$, and $P = \conv (S_1 \cup S_2)$.
\item[(ii)] There are regular simplices $S_1, S_2$, of dimensions $\dim S_1 = k_1 \leq k$ and $\dim S_2 = k_2 \leq k$ and lying in orthogonal subspaces that meet at the common center $q_1$ of both simplices, and a regular $(d-k_1-k_2)$-dimensional simplex $Q$ with center $q_2$ such that $Q$ is perpendicular to the segment $[q_1,q_2]$, and $P= \conv (S_1 \cup S_2 \cup Q)$.
\item[(iii)] $P$ is a regular simplex of unit volume.
\end{itemize}
\end{thm}

\section{Applications for measures with rotationally symmetric density functions}\label{sec:app}

Theorems~\ref{thm:circumscribed_dp1} and \ref{thm:circumscribed_dp2} are immediate consequences of the following, more general theorems, which can be proved by the same idea (for the proof of Theorem~\ref{thm:density1}, see also the proof of \cite[Lemma 9]{BL22}).

\begin{thm}\label{thm:density1}
Let $D$ be an open ball centered at $o$, or let $D = \Eu^d$. Let $\rho : [0, R_0)$ be a nonnegative function, where $R_0 \in \Re \cup \{ \infty\}$ denotes the radius of $D$. Let $S_{\mathrm{reg}}  \subset D$ be a regular $d$-dimensional simplex centered at $o$.
\begin{itemize}
\item[(i)] If $\rho$ is a decreasing function, then for any simplex $S \subset D$ with $\cirr(S) = \cirr(S_{\mathrm{reg}})$, we have
\[
\int_{S} \rho(||x||) \, dx \geq \int_{S_{\mathrm{reg}}} \rho(||x||) \, dx.
\]
Furthermore, if $\rho$ is strictly decreasing, then equality holds if and only if $S$ is congruent to $S_{\mathrm{reg}}$, and its center is $o$.
\item[(ii)] If $\rho$ is increasing, then for any $d$-dimensional simplex $S \subset D$ with $\ir(S) = \ir(S_{\mathrm{reg}})$, we have
\[
\int_{S} \rho(||x||) \, dx \leq \int_{S_{\mathrm{reg}}} \rho(||x||) \, dx.
\]
Furthermore, if $\rho$ is strictly increasing, then equality holds if and only if $S$ is congruent to $S_{\mathrm{reg}}$, and its center is $o$.
\end{itemize}
\end{thm}

\begin{thm}\label{thm:density2}
Let $D$ be an open ball centered at $o$, or let $D = \Eu^d$. Let $\rho : [0, R_0)$ be a nonnegative function, where $R_0 \in \Re \cup \{ \infty\}$ denotes the radius of $D$. Let $P \subset D$ be a convex polytope with $d+2$ vertices.
\begin{itemize}
\item[(i)] If $\rho$ is a decreasing function, then there is a convex polytope $Q = \conv (S_1 \cup S_2)$ of circumradius equal to $\cirr(P)$, where $S_1$ and $S_2$ are regular simplices centered at $o$ such that the subspaces spanned by $S_1$ and $S_2$ are orthogonal complements of each other, and
\[
\int_{P} \rho(||x||) \, dx \geq \int_{Q} \rho(||x||) \, dx.
\]
\item[(ii)] If $\rho$ is a decreasing function, then there is a convex polytope $Q = \conv (S_1 \cup S_2)$ of inradius equal to $\ir(P)$, where $S_1$ and $S_2$ are regular simplices centered at $o$ such that the subspaces spanned by $S_1$ and $S_2$ are orthogonal complements of each other, and
\[
\int_{P} \rho(||x||) \, dx \leq \int_{Q} \rho(||x||) \, dx.
\]
\end{itemize}
\end{thm}

We remark that since by central projection, an open hemisphere in $\Sph^d$ can be represented as the Euclidean space $\Eu^d$ equipped with a strictly decreasing rotationally symmetric density function, Theorems~\ref{thm:circumscribed_dp1} and \ref{thm:circumscribed_dp2} can be also applied for spherical polytopes. As an example for another application, we present Corollary~\ref{cor:Gaussian}.

\begin{cor}\label{cor:Gaussian}
Recall that the density function of the standard Gaussian measure on $\Eu^d$ is $\rho(||x||)$ with $\rho(\tau)= \frac{1}{(2\pi)^{d/2}} e^{-\tau^2/2}$. Since $\rho(\tau)$ is a strictly decreasing function of $\tau$, by Theorem~\ref{thm:density1} we have that among simplices in $\Eu^d$ with a given circumradius, the regular simplex centered at the origin has maximal Gaussian measure. By Theorem~\ref{thm:density2}, we obtain a similar statement for the Gaussian measures of polytopes in $\Eu^d$ with $d+2$ vertices and a given circumradius.
\end{cor}

\section{Concluding remarks and open problems}\label{sec:remarks}

Motivated by Theorem~\ref{thm:density1}, it is a natural question to ask if for any measure generated by a rotationally symmetric density function $\rho$, if $B$ is a ball centered at $o$, then among simplices $S$ contained in $B$, $\int_S \rho (||x||) d x$ is maximal if $S$ is a regular simplex inscribed in $B$. Our next example shows that this is not true in general.

\begin{ex}
Let $B$ be the closed unit disk centered at $o$, let $0< \varepsilon \leq 1$, and set
\[
\rho(\tau) = \left\{
\begin{array}{l}
1, \hbox{ if } \tau \in [1-\varepsilon,1],\\
0, \hbox{ otherwise.}
\end{array}
\right.
\]
Let $T_{\mathrm{reg}}$ denote a regular triangle inscribed in $B$, and let $T$ denote an isosceles triangle inscribed in $B$ such that the length of its altitude belonging to its base is $\varepsilon$. An elementary computation shows that $\int_{T_{\mathrm{reg}}} \rho(||x||) dx = \Theta (\varepsilon^{2})$, and
$\int_{T} \rho(||x||) dx = \Theta (\varepsilon^{3/2})$. Thus, if $\varepsilon$ is sufficiently small, $\int_{T_{\mathrm{reg}}} \rho(||x||) dx < \int_{T} \rho(||x||) dx$.
\end{ex}

We could not prove that among simplices in $\Sph^d$ containing a given ball $B$, the regular ones circumscribed about $B$ have minimal volume.
Nevertheless, we show that an affirmative answer to this problem implies a spherical variant of the so-called Simplex Mean Width Conjecture (see e.g. \cite{Litvak}), stating that among simplices contained in a given ball $B$ in $\Eu^d$, the regular ones inscribed in $B$ have maximal mean width. Before introducing it, we remark that the mean width of a convex body is proportional to its first intrinsic volume, and note that, apart from the equality case, by the property described in Remark~\ref{rem:SMWC} an affirmative answer to Problem~\ref{prob:smw} implies the Simplex Mean Width Conjecture by a standard limit argument.

\begin{rem}\label{rem:SMWC}
A detailed description of possible variants of intrinsic volumes of convex bodies in $\Sph^d$ can be found in \cite{GHS}. One of these, based on the geometric observation that the first intrinsic volume of a convex body in $\Eu^d$ is proportional to the total rotation invariant measure of the hyperplanes intersecting it, is defined as follows.
Let $\mathcal{S}$ denote the space of spherical hyperplanes in $\Sph^d$, let $\mu$ denote the rotation invariant probability measure on $\mathcal{S}$, and let $\chi(\cdot)$ define Euler characteristic. For any convex body in $\Sph^d$, set
\[
U_1(K) = \frac{1}{2} \int_{\mathcal{S}} \chi(K \cap S) \, d \nu(S).
\]
By (20) of \cite{GHS}, for any spherically convex set $K \subset \Sph^d$, we have
\[
U_1(K) = \frac{1}{2} - \frac{\vol_d(K^*)}{\vol_d(\Sph^d)}
\]
where $K^*$ denotes the polar of $K$. Thus, in the family of simplices in $\Sph^d$ contained in a given ball $B \subset \Sph^d$, the functional $U_1$ attains its maximum at $S$ if and only if among the simplices containing $B^*$, the simplex $S^*$ has minimal volume.
\end{rem}

\begin{prob}\label{prob:smw}
Let $d \geq 4$. Prove or disprove that for any ball $B \subset \Sph^d$, among the simplices containing $B$ the regular ones circumscribed about $B$ have minimal volume.
\end{prob}

\end{document}